\documentclass[12pt]{amsart}

\usepackage{tikz-cd}
\usepackage{psfrag}
\usepackage{color}
\usepackage{tikz}
\usetikzlibrary{matrix,arrows}
\usepackage{graphicx,graphics}
\usepackage{fullpage,amssymb,amsfonts,amsmath,amstext,amsthm,amscd,verbatim,enumerate}
\usepackage[T1]{fontenc}

\begin{document}

\newtheorem{theorem}{Theorem}[section]
\newtheorem{result}[theorem]{Result}
\newtheorem{fact}[theorem]{Fact}
\newtheorem{example}[theorem]{Example}
\newtheorem{conjecture}[theorem]{Conjecture}
\newtheorem{lemma}[theorem]{Lemma}
\newtheorem{proposition}[theorem]{Proposition}
\newtheorem{corollary}[theorem]{Corollary}
\newtheorem{facts}[theorem]{Facts}
\newtheorem{props}[theorem]{Properties}
\newtheorem*{thmA}{Theorem A}
\newtheorem{ex}[theorem]{Example}
\theoremstyle{definition}
\newtheorem{definition}[theorem]{Definition}
\newtheorem{remark}[theorem]{Remark}
\newtheorem*{defna}{Definition}

\newcommand{\notes} {\noindent \textbf{Notes.  }}
\newcommand{\note} {\noindent \textbf{Note.  }}
\newcommand{\defn} {\noindent \textbf{Definition.  }}
\newcommand{\defns} {\noindent \textbf{Definitions.  }}
\newcommand{\x}{{\bf x}}
\newcommand{\z}{{\bf z}}
\newcommand{\B}{{\bf b}}
\newcommand{\V}{{\bf v}}
\newcommand{\T}{\mathbb{T}}
\newcommand{\Z}{\mathbb{Z}}
\newcommand{\Hp}{\mathbb{H}}
\newcommand{\D}{\mathbb{D}}
\newcommand{\R}{\mathbb{R}}
\newcommand{\N}{\mathbb{N}}
\renewcommand{\B}{\mathbb{B}}
\newcommand{\C}{\mathbb{C}}
\newcommand{\ft}{\widetilde{f}}
\newcommand{\dt}{{\mathrm{det }\;}}
 \newcommand{\adj}{{\mathrm{adj}\;}}
 \newcommand{\0}{{\bf O}}
 \newcommand{\av}{\arrowvert}
 \newcommand{\zbar}{\overline{z}}
 \newcommand{\xbar}{\overline{X}}
 \newcommand{\htt}{\widetilde{h}}
\newcommand{\ty}{\mathcal{T}}
\renewcommand\Re{\operatorname{Re}}
\renewcommand\Im{\operatorname{Im}}
\newcommand{\tr}{\operatorname{Tr}}
\newcommand{\Stab}{\operatorname{Stab}}
\newcommand\capac{\operatorname{cap}}
\newcommand\diam{\operatorname{diam}}
\renewcommand\mod{\operatorname{mod}}

\newcommand{\ds}{\displaystyle}
\numberwithin{equation}{section}

\renewcommand{\theenumi}{(\roman{enumi})}
\renewcommand{\labelenumi}{\theenumi}

\title{Quasiregular semigroups with examples}

\author{Alastair Fletcher}
\address{Department of Mathematical Sciences, Northern Illinois University, DeKalb, IL 60115-2888. USA}
\email{fletcher@math.niu.edu}

\thanks{This work was supported by a grant from the Simons Foundation (\#352034, Alastair Fletcher).}

\begin{abstract}
Rational semigroups were introduced by Hinkkanen and Martin as a generalization of the iteration of a single rational map. There has subsequently been much interest in the study of rational semigroups. Quasiregular semigroups were introduced shortly after rational semigroups as analogues in higher real dimensions, but have received far less attention. Each map in a quasiregular semigroup must necessarily be a uniformly quasiregular map. While there is a completely viable theory for the iteration of uniformly quasiregular maps, it is a highly non-trivial matter to construct them. In this paper, we study properties of the Julia and Fatou sets of quasiregular semigroups and, equally as importantly, give several families of examples illustrating some of the behaviours that can arise.
\end{abstract}

\maketitle

\section{Introduction}

Many aspects of complex dynamics have been carried over to the setting of rational semigroups.
A  rational semigroup $G$ is a collection of rational maps on the Riemann sphere $\overline{\C}$ where the binary operation is function composition. The semigroup is generated by $g_1,g_2,\ldots$ and we write $G = \left <g_1,g_2,\ldots \right >$. The study of rational semigroups is a natural generalization of the iteration theory of a single rational map $g$. In this case, one can ask for the set on which the iterates of $g$ behave stably and the set on which the iterates behave chaotically. These sets are the familiar Fatou set $F(g)$ and Julia set $J(g)$ respectively.

One can define the Fatou and Julia sets for a rational semigroup $G$ via normal families in the same way as for a rational map.
We will give precise definitions below. The study of rational semigroups was initiated by Hinkkanen and Martin \cite{HM} and has been the subject of much study since, see the recent paper \cite{JS} and the references therein.

Currently, there is a great deal of interest in the iteration theory of quasiregular mappings in $\R^n$.
Quasiregular mappings, or mappings of bounded distortion, are the correct generalization of holomorphic functions into $\R^n$, for $n\geq 2$, if one wants to have an interesting function theory. Quasiregular dynamics is then the natural counterpart of complex dynamics in higher dimensions, see \cite{Berg} for an introduction to this theory. 

Here, we will be mostly interested in uniformly quasiregular mappings (henceforth abbreviated to uqr mappings), that is, those for which there is a uniform bound on the distortion of the iterates. Normal family machinery in this setting allows one to define a Fatou set and Julia set for uqr mappings and, moreover, these partition $\R^n$. One can then study quasiregular semigroups generated by, necessarily, uqr maps. This was initiated in \cite{IM} but, as far as the author is aware, there has been little systematic study of the properties of quasiregular semigroups and their Julia and Fatou sets.

In this paper, our purpose is twofold: first, to show that many of the properties that hold for rational semigroups also hold for quasiregular semigroups. Second, we will construct several examples of quasiregular semigroups that illustrate some of the possibilities that can occur. Perhaps one of the reasons that quasiregular semigroups have not received much attention is that it is highly non-trivial to construct even uqr maps in higher dimensions. In dimension two, every uqr map is a quasiconformal conjugate of a holomorphic map, but no such result is true in higher dimensions. Consequently, it is not obvious that there are many quasiregular semigroups to apply the theory to. We will exhibit several families of quasiregular semigroups.

The outline of the paper is as follows. In section two, we will cover preliminary material on quasiregular mappings and tools we will need. In section three, we will discuss some of the basic properties of quasiregular semigroups. In section four, we will prove that given a condition that is satisfied by, for example, finitely generated quasiregular semigroups, the Julia set is uniformly perfect. In section five, we will exhibit several classes of quasiregular semigroups, mostly based on families of solutions to the Schr\"oder equation.

\section{Preliminaries}

\subsection{Quasiregular mappings}

A {\it quasiregular mapping} in a domain $U\subset \R^n$ for $n\geq 2$ is a continuous mapping in the Sobolev space $W^1_{n,loc}(U)$ where there is a uniform bound on the distortion, that is, there exists $K\geq 1$ such that
\[|f'(x)|^n \leq KJ_f(x)\]
almost everywhere in $U$. The minimum such $K$ for which this inequality holds is called the {\it outer dilatation} and denoted by $K_O(f)$. As a consequence of this, there is also $K' \geq 1$ such that 
\[J_f(x) \leq K' \inf_{|h|=1}|f'(x)h|^n\]
holds almost everywhere in $U$. The minimum such $K'$ for which this inequality holds is called the {\it inner dilatation} and denoted by $K_I(f)$. If $K= \max \{K_O(f), K_I(f) \}$, then $K=K(f)$ is the maximal dilatation of $f$. A $K$-quasiregular mapping is a quasiregular mapping for which $K(f) \leq K$. An injective quasiregular mapping is called quasiconformal.

We will often identify $\R^n \cup \{ \infty \}$ with the sphere $S^n$ and use the chordal metric $\chi$. If $A$ is a M\"obius map sending the point at infinity to $0$, then we can consider quasiregular mappings that are either defined at infinity, or have poles, by pre- or post-composing by $A$ and applying the above condition. Such mappings are sometimes called quasimeromorphic, but we will keep the nomenclature quasiregular on $S^n$.

The branch set $B_f$ consists of the points where $f$ is not locally injective. The branch set for a non-injective quasiregular mapping is always non-empty in dimension at least three. 
We refer to \cite{Rickman} for many more details on the foundations of quasiregular mappings, but we note here the following analogue of Picard's Theorem due to Rickman.

\begin{theorem}[\cite{Rickman}, Theorem IV.2.1]
\label{rickman}
For every $n\geq 2$ and $K\geq 1$, there exists a positive integer $q=q(n,K)$ which depends only on $n$ and $K$, such that the following holds. Every $K$-quasiregular mapping $f: \R^n \to S^n \setminus \{a_1,\ldots, a_m \}$ is constant whenever $m\geq q$ and $a_1,\ldots, a_m$ are distinct points in $S^n$.
\end{theorem}

We will also require the following result showing that quasiregular maps have bounded linear distortion. Given $f:U \to \R^n$ a $K$-quasiregular map and $x \in U$, for $r < d(x , \partial U)$ define
\begin{equation}
\label{eq:Ll} 
L(x,r) = \max_{|y-x| = r} |f(y) - f(x)|, \quad l(x,r) = \min_{|y-x|=r} |f(y)-f(x)|.
\end{equation}
Then the linear distortion of $f$ at $x$ is
\[ H(x,f) = \limsup_{r\to 0 } \frac{L(x,r)}{l(x,r)}.\]

\begin{theorem}[\cite{Rickman}, Theorem II.4.3]
\label{thm:linear}
Let $f:U \to \R^n$ be a non-constant $K$-quasiregular map. Then for every $x\in U$, there exists $C$ depending only on the product $i(x,f)K_O(f)$ and $n$ so that
\[ H(x,f) \leq C.\]
\end{theorem}

\subsection{Uqr maps and normal families}

For $m\geq 1$, we write $f^m$ for the $m$-fold iterate of $f$. A mapping is called {\it uniformly $K$-quasiregular}, or $K$-uqr for short, if $K(f^m) \leq K$ for all $m\geq 1$.

The reason uqr mappings are closely related to holomorphic functions in the plane is the following analogue of Montel's Theorem.

\begin{theorem}[\cite{Miniowitz}]
\label{mini}
Let $\mathcal{F}$ be a family of $K$-quasiregular mappings in a domain $G\subset S^n$ and let $q=q(n,K)$ be Rickman's constant from Theorem \ref{rickman}. If there exist distinct points $a_1,\ldots, a_q \in S^n$ such that $f(G) \cap \{a_1,\ldots, a_q\} = \emptyset$ for all $f\in \mathcal{F}$, then $\mathcal{F}$ is a normal family.
\end{theorem}

We recall that $\mathcal{F}$ is a normal family of $K$-quasiregular mappings if it is pre-compact in the space $C(G,S^n)$ of continuous functions from a domain $G \subset S^n$ into $S^n$ in the topology of local uniform convergence. Moreover, then every limit function is either a constant or a $K$-quasiregular mapping by \cite[Theorem VI.8.6]{Rickman}.

\subsection{Quasiregular semigroups}

We first recall the definition of a quasiregular semigroup from \cite{IM}.

\begin{definition}
A $K$-quasiregular semigroup $G$ on $S^n$, for $n\geq 2$, is a family of mappings $g:S^n \to S^n$ so that each $g\in G$ is $K$-quasiregular, non-injective and the family is closed under composition. A quasiregular semigroup is a $K$-quasiregular semigroup for some $K\geq 1$. If $G$ is generated by $\{ g_i : i\in I \}$, then we write $G = \left < g_i : i\in I \right >$.
\end{definition}

Applying Theorem \ref{mini}, we see that $G$ partitions $S^n$ into two sets: one of stable behaviour and one of chaotic behaviour.

\begin{definition}
Let $G$ be a quasiregular semigroup on $S^n$, for $n\geq 2$. Then:
\begin{enumerate}[(i)]
\item $x\in S^n$ is said to be in the Fatou set of $G$, denoted by $F(G)$, if there exists a neighbourhood $U$ of $x$ so that $G$ restricted to $U$ is a normal family;
\item $x\in S^n$ is said to be in the Julia set of $G$, denoted by $J(G)$, is there is no neighbourhood $U$ of $x$ on which $G$ restricted to $U$ is a normal family.
\end{enumerate}
\end{definition}

It is clear from the definition that $F(G)$ is open and $J(G)$ is closed.
We remark that in the theory of rational semigroups, the Fatou set is often denoted by $N(G)$. In this paper, we will keep in line with current notation in complex dynamics and use $F(G)$ to denote this set.

\begin{example}
If $f:S^n \to S^n$ is a uqr mapping, then $G = \{ f, f^2, f^3 ,\ldots \}$ is a quasiregular semigroup generated by the single map $f$. We then write $G = \left < f \right >$, $J(\left <f \right >) = J(f)$ and $F(\left <f \right >) = F(f)$.
\end{example}

\subsection{Modulus, capacity and uniformly perfect sets}

We refer to \cite{Vuorinen} for more details on the notions introduced in this subsection.

Write $\chi$ and $d$ for the chordal and Euclidean distances on $S^n$
and $\R^{n}$ respectively.
The chordal distance is normalized so that $\chi(x,y)\leq 1$ for all $x,y \in S^n$ with equality if and only if $x$ and $y$ are antipodal points. Open balls in the chordal metric will be denoted by $B_{\chi}(x,r)$ for $x\in S^n$ and $r>0$. Open balls in $\R^n$ in the Euclidean metric will be denoted by $B(x,r)$. We also write $\mathbb{B}^n=B(0,1)$.

For sets $E$ and $F$ in $S^n$, we write $\chi(E)$ for the chordal diameter of $E$ and $\chi(E,F)$ for the chordal distance between $E$ and $F$.
If $E$ consists of one point $x \in S^n$, we write $\chi(x,F)$.
Similarly, for sets $E$ and $F$ in $\R^{n}$, we write $d(E)$ for the Euclidean diameter of $E$ and $d(E,F)$ for the Euclidean distance between $E$ and $F$. 
Denote by $A(x,r,s)$ the Euclidean annulus $\{y\in \R^{n} : r<d(y,x) <s\}$. The chordal annulus
$A_{\chi}(x,r,s)$ is defined analogously.

A domain $R\subset S^n$ is called a ring domain if $S^n \setminus R$ has exactly two components. If the two components are $C_{0}$ and $C_{1}$, then we write $R = R(C_{0},C_{1})$.

Given two sets $E$ and $F$, we write $\Delta (E,F ; V)$ for the family of paths with one end-point in $E$, the other end-point in $F$, and which are contained in $V$. When $V= S^n$, we write $\Delta(E,F) = \Delta (E,F;S^n)$.

The $n$-modulus $M(\Gamma)$ of a path family $\Gamma$ is defined by
\begin{equation*}
M(\Gamma) = \inf \int _{\R^{n}} \rho ^{n} \: dm,
\end{equation*}
where $m$ denotes $n$-dimensional Lebesgue measure, and the infimum is taken
over all non-negative Borel measurable functions $\rho$ such that
\begin{equation*}
\int _{\gamma} \rho \: ds \geq 1
\end{equation*}
for each locally rectifiable curve $\gamma \in \Gamma$.

The conformal modulus of a ring domain $R(C_{0},C_{1})$ is defined by
\begin{equation}
\label{moddef}
\mod R(C_{0},C_{1}) = \left ( \frac{ M(\Delta(C_{0},C_{1}))}{\omega _{n-1}} \right ) ^{1/(1-n)},
\end{equation}
where $\omega_{n-1}$ is the $(n-1)$-dimensional area of the unit $(n-1)$-sphere. 
If $R_{0},R_{1}$ are two ring domains with $R_{0} \subset R_{1}$, then
\begin{equation}
\label{modeq1}
\mod R_{0} \leq \mod R_{1}.
\end{equation}
The capacity of a ring domain $R(C_{0},C_{1})$ is defined to be 
\begin{equation}
\label{capdef}
\capac R = M(\Delta(C_{0},C_{1})).
\end{equation}

We recall that a closed set $X \subset S^n$ is perfect if and only if it contains no isolated points. Uniformly perfect sets are perfect in a quantitative way. More precisely, we have the following:

\begin{definition}
A closed set $X \subset S^n$ containing at least $2$ points is called $\alpha$-uniformly perfect if there is no ring domain $R \subset S^n$ separating $X$ such that $\mod R > \alpha$. Further,
$X$ is called uniformly perfect if it is $\alpha$-uniformly perfect for some $\alpha >0$.
\end{definition}

We can pass from general ring domains to round ring domains as follows, see for example \cite[Lemma 2.6 and Corollary 2.7]{FN}.

\begin{lemma}
\label{lem:chordal}
A closed set $X \subset S^n$ is uniformly perfect if and only if the moduli of the chordal annuli separating $X$ are bounded from above. Moreover, if $X$ is not uniformly perfect, then there is a sequence of chordal annuli $A_m$ with centres at $x_m \in X$ and $\mod(A_m) \to \infty$.
\end{lemma}

\section{Basic properties of quasiregular semigroups}

In this section, we note some properties of quasiregular semigroups that will be in direct analogy of those for rational semigroups as outlined in \cite{HM}. Throughout this section, $G$ is a quasiregular semigroup.

Our first result is clear from the normal family definition.

\begin{proposition}
\label{prop:1}
If $g\in G$, then $F(G) \subseteq F(g)$ and $J(g) \subseteq J(G)$.
\end{proposition}

A set $X$ is called {\it forward invariant} under $G$ if $g(X) \subseteq X$ for all $g\in G$, {\it backward invariant} if $g^{-1}(X) \subseteq X$ for all $g\in G$ and {\it completely invariant} if $X$ is both forward and backward invariant. It is well-known that the Julia set and Fatou set of a uqr map are completely invariant and we will see below that the same need not be true for quasiregular semigroups. We do however have the following two results.

\begin{proposition}
\label{prop:2}
The Fatou set $F(G)$ is forward invariant and the Julia set $J(G)$ is backward invariant.
\end{proposition}

\begin{proposition}
\label{prop:3}
If $G = \left < g_1,\ldots, g_m \right >$, then
\[ F(G) = \bigcap _{i=1}^m g_i^{-1}( F(G) ) ,\quad J(G) = \bigcup_{i=1}^m g_i^{-1} ( J(G)).\]
\end{proposition}

The proof of these results are identical to those for rational semigroups given in \cite[Lemma 1.1.4]{Sumi} and are omitted.
A point $x\in S^n$ is called {\it exceptional} if the backward orbit
\[ O^-(x) = \{ y : g(y)=x \text{ for some } g\in G \}\]
is finite. The set of exceptional points is denoted by $\mathcal{E}(G)$. If $G$ is finitely generated, then $\mathcal{E}(G) \subset F(G)$.

\begin{proposition}
\label{prop:5}
If $G$ is a $K$-quasiregular semigroup in $S^n$, then
the number of elements of $\mathcal{E}(G)$ is at most $q-1$, where $q=q(n,K)$ is Rickman's constant from Theorem \ref{rickman}.
\end{proposition}

\begin{proof}
If $\mathcal{E}(G)$ contains at least $q$ points, then the family $G$ applied to $S^n \setminus \mathcal{E}(G)$ must be normal. This means $J(G)$ contains finitely many elements. However, since each $g\in G$ is not injective, by \cite[Lemma 5.3]{FNsphere}, $J(g)$ contains infinitely many elements. This contradicts Proposition \ref{prop:1}.
\end{proof}

\begin{proposition}
\label{prop:4}
If $x \in S^n \setminus \mathcal{E}(G)$, then $J(G) \subseteq \overline{O^-(x)}$.
\end{proposition}

\begin{proof}
Let $y\in J(G)$ and $U$ any neighbourhood of $y$. Suppose there is no $w \in U' = U \setminus \{ y \}$ for which $g(w) = x$ for some $g\in G$. Then for any $u \in O^-(x)$, there is no $g\in G$ with $g(w) = u$. Since $G$ is not normal in $U$ because $y\in J(f)$, $O^-(x)$ must contain at most $q-1$ points. Then $x$ is exceptional, which was assumed to not be the case.
\end{proof}

\begin{proposition}
\label{prop:6}
The Julia set $J(G)$ is the smallest closed backward invariant set containing at least $q$ points.
\end{proposition}

\begin{proof}
The complement of a backward invariant closed set containing at least $q$ points is a forward invariant open set omitting at least $q$ points. Hence by Theorem \ref{mini} the complement must be contained in the Fatou set.
\end{proof}

\begin{proposition}
\label{prop:7}
The Julia set $J(G)$ is a perfect set.
\end{proposition}

\begin{proof}
Since each element $g\in G$ is non-injective, $J(g)$ contains infinitely many points for each $g\in G$. Therefore, by Proposition \ref{prop:1}, at least one point of $J(G)$ is not exceptional and consequently $J(G)$ contains infinitely many elements by backward invariance. If we let $J'$ denote the derived set of $J(G)$, that is, the set of accumulation points of $J(G)$, then $J'$ is non-empty, closed and backward invariant. The set $J'$ cannot be finite since it would then be exceptional and so $J' = J(G)$ by Proposition \ref{prop:6}. It follows that $J(G)$ cannot have any isolated points.
\end{proof}

\section{Further results on quasiregular semigroups}

In this section, we prove some results on quasiregular semigroups where we have to approach the proofs in at least a somewhat different manner to the corresponding proofs in rational semigroup theory. This will also require some refinements of results in uqr dynamics. The main aim of this section is to prove that if $G$ is a quasiregular semigroup satisfying a certain uniform H\"older condition, then $J(G)$ is uniformly perfect.

We recall from \cite{HMM} that a fixed point of a uqr map $f$ is called {\it superattracting} if it is contained in the branch set of $f$. It follows from the local H\"older estimates for quasiregular maps \cite[Theorem III.4.7]{Rickman} (see \cite[p.87]{HMM}) that if $x_0=0$ is superattracting for a $K$-uqr map $f$ with local index $i=i(0,f) \geq 2$, then for every $k\in \N$, there is a neighbourhood $U$ of $0$ and a constant $C>0$ so that
\begin{equation}
\label{eq:holder}
|f^k(x)| \leq C|x|^{\mu},
\end{equation}
for all $x\in U$, where $\mu = (i^k/K)^{1/(n-1)}$. Hence for $i^k>K$ large enough, there exists $r_0>0$ so that if $r<r_0$ then $\overline{f^k(B(0,r))} \subset B(0,r)$. For our purposes, we need to know that this is true for $k=1$, but as far as we are aware, this result is not available in the literature. We observe that this property is one way to define attracting fixed points for uqr maps if $i(x_0,f) = 1$.

In \cite{HMM}, the idea of generalized derivatives is introduced to classify fixed points of uqr maps. Again assuming the fixed point is $x_0 =0$, if we define
\begin{equation}
\label{eqfj} 
f_{\lambda}(x) = \lambda f \left ( \frac{x}{\lambda} \right )
\end{equation}
then generalized derivatives arise as limits $f_{\lambda_j}$
through sequences $\lambda_j \to \infty$. It follows from Theorem \ref{mini} and distortion estimates when $i(0,f)=1$ (see \cite[p.87]{HMM}) that we can find a subsequence along which $f_{\lambda_j}$ converges to a quasiregular map.

\begin{lemma}
\label{lem:superattr}
Let $U \subset S^n$ be a domain and $f:U \to S^n$ a $K$-uqr map and $x_0$ a superattracting fixed point of $f$ with $x_0 \in U$. Then there exists $r_0>0$ so that if $r<r_0$ then $\overline{f(B_{\chi}(x_0,r) ) } \subset B_{\chi}(x_0,r)$.
\end{lemma}

\begin{proof}
There is no loss of generality in assuming $x_0 = 0$, and then we can pass to the Euclidean metric and use \eqref{eq:holder}.
By \eqref{eq:holder}, we can find $C>0$, $r_0>0$ and $k\in \N$ large enough so that 
\begin{equation}
\label{eq:fk} 
|f^k(x)| \leq C|x|^2,
\end{equation}
for $|x| < r_0$. We then consider generalized derivatives for $f^k$ at $0$. In fact, we have for $|x| < \lambda r_0$ that
\[ | (f^k)_{\lambda} (x) |  = \lambda \left | f^k \left (  \frac{x}{\lambda} \right )  \right | \leq \frac{C|x|^2}{\lambda}.\]
Consequently, through any sequence $\lambda_j \to \infty$, the limit of $(f^k)_{\lambda_j}$ must be identically $0$.

Next, we need to show that we can consider generalized derivatives of $f$ itself. To that end, first note that there exists $r_1>0$ so that if $r<r_1$, then $l(0,r)\leq r$, recalling \eqref{eq:Ll}. If this is not true, then there is a sequence $r_j \to 0$ with $l(0,r_j) > r_j$. However, this contradicts \eqref{eq:fk} because then $ \overline{ f(B(0,r_j) ) } \supset B(0,r_j)$. Next, apply Theorem \ref{thm:linear} to find $r_2>0$ so that if $r<r_2$ then 
\[ \frac{L(0,r)}{l(0,r)} \leq 2C',\]
where $C'$ depends only on $n$ and $i(0,f)K_O(f)$. In particular, this implies that if $r< \min \{r_1,r_2 \}$ then $|f(x)| \leq 2C'|x|$ for $|x|\leq r$. This uniform Lipschitz estimate implies that $f$ has generalized derivatives at $0$.

Since $(f_{\lambda})^k = (f^k)_{\lambda}$, it follows that if $\varphi$ is any generalized derivative of $f$, then $\varphi ^k$ must be identically $0$. Consequently, so must $\varphi$. 

Now suppose that the result claimed in the lemma is false, that is, we can find a sequence $r_j \to 0$ and $x_j$ with $|x_j| = r_j$ and $|f(x_j)| \geq |x_j|$. Consider $f_{\lambda_j}$, where $\lambda_j = r_j^{-1}$ and set $y_j = \lambda_jx_j$. Since $|y_j|=1$, we can assume that we have passed to a subsequence where both $y_j$ converges to some point of the unit sphere, and $f_{\lambda_j}$ converges to a generalized derivative $\varphi$ of $f$. But then
\[ |f_{\lambda_j}(y_j)| = \lambda_j | f(x_j)| \geq \lambda_j|x_j|  = 1.\]
Consequently $\varphi$ is non-constant, contradicting the fact it must be identically zero.
\end{proof}

We next require the following uqr analogue of \cite[Theorem 6.9.4]{Beardon}.

\begin{lemma}
\label{lem:blowup}
Let $f:S^n \to S^n$ be a $K$-uqr map that is not injective, let $U$ be a domain that meets $J(f)$ and let $V$ be a compact set that does not meet $\mathcal{E}(f)$. Then there exists $M\in \N$ so that for $m\geq M$, $f^m(U) \supset V$.
\end{lemma}

\begin{proof}
We know that $\mathcal{E}(f)$ is a finite set. By the discussion on \cite[p. 83]{HMM}, each $x\in \mathcal{E}(f)$ is a fixed point of $f^m$ for some $m$ with $f^{-m}(x) = \{ x\}$. It follows that each element of $\mathcal{E}(f)$ is part of a superattracting cycle. By applying Lemma \ref{lem:superattr}, we can find a neighbourhood $N$ of $\mathcal{E}(f)$ so that $f(N) \subset N$. If $\mathcal{E}(f)$ is empty, we can take $N = \emptyset$, otherwise we can assume $N$ is small enough so that $N$ and $V$ are disjoint.

As $U$ meets $J(f)$, by \cite[Proposition 3.2]{HMM} we can find a neighbourhood $U'$ of some point in $J(f)$ so that $U '\subset U$ and there exists $q\in \N$ so that $f^{mq}(U')$ is an increasing sequence of open sets exhausting $S^n \setminus \mathcal{E}(f)$.

Let $Q = S^n \setminus N$. Then $Q$ is compact, $V \subset Q$ and since $f(N) \subset N$, we have $f(Q) \supset Q$. Since $f^{mq}(U')$ is increasing, there exists $t$ so that $Q \subset f^t(U')$. Then for $m\geq t$,
\[ f^m(U) \supset f^m(U') \supset f^{m-t}f^t(U') \supset f^{m-t}(Q) \supset Q \supset V.\]
\end{proof}

Our next lemma shows that the Julia sets of certain compositions accumulate on the pre-image of a given Julia set, and may even be new for rational functions.

\begin{lemma}
\label{lem:prop8}
Let $g,h$ be uqr maps $S^n \to S^n$, for $n\geq 2$, so that $G=<g,h>$ is a quasiregular semigroup.
For any $\epsilon >0$, there exists $M \in \N$ so that if $m\geq M$ and $x\in g^{-1}(J(h))$, there exists $y \in J(h^mg)$ with $\chi (x,y) < \epsilon$.
\end{lemma}

\begin{proof}

Given $X \subset S^n$ and $\delta >0$, denote by $N_{\delta}(X)$ the open $\delta$-neighbourhood of $X$ in the chordal metric. Denote by $U_{\delta}$ the set $N_{\delta}( g^{-1}(J(h)) )$.

We will first assume that $g^{-1}(J(h)) \cap \mathcal{E}(h) = \emptyset$. We may suppose that $\delta >0$ is small enough that $\overline{U_{\delta} } \cap \mathcal{E}(h) = \emptyset$.
Given $r>0$, find finitely many points $x_1,\ldots, x_k \in S^n$ so that $B_i = B_{\chi}(x_i,r)$ forms an open cover of the compact set $g^{-1}(J(h))$. Since each $g(B_i)$ is an open set intersecting $J(h)$, by Lemma \ref{lem:blowup} we can find $M_1 \in \N$ so that
\[ \overline{U_{\delta}} \subset (h^mg) (B_i)\]
for $m\geq M_1$ and $i=1,\ldots, k$. If $m\geq M_1$, we have $J(h^mg) \subset (h^mg)(B_i)$ for $i=1,\ldots,k$. By complete invariance, it follows that for $i=1,\ldots, k$ there exist points of $J(h^mg)$ in $B_i$. Since this argument applies for arbitrarily small $r$, it follows that given $\epsilon >0$, there exists $M\in \N$ so that if $m\geq M$,
\[ g^{-1}(J(h)) \subset N_{\epsilon}( J(h^mg) ).\]

Next, suppose that $g^{-1}(J(h)) \cap \mathcal{E}(h) \neq \emptyset$. Since $\mathcal{E}(h)$ is a finite set and $J(h)$ is infinite, we can find $y \in J(h)$ and $z \in g^{-1}(y)$ so that $z \notin \mathcal{E} (h)$. We may then find $\eta>0$ so that the component $E$ of $g^{-1}( B_{\chi}(y,\eta))$ containing $z$ satisfies $\overline{E} \cap \mathcal{E}(h) = \emptyset$.
Since $B_{\chi}(y,\eta) \cap J(h) \neq \emptyset$, there exists $M_2 \in \N$ so that if $m \geq M_2$, then $E \subset h^m ( B_{\chi} (y,\eta) )$. 

We find a sequence of nested closed sets, starting with $E_0 = \overline{E}$ and finding $E_1 \subset E_ 0$ so that if $m\geq M_2$ is fixed, then $(h^mg)(E_1) = E_0$. Continuing inductively, we find $E_{j+1} \subset E_j$ such that $(h^mg)(E_{j+1}) = E_j$. If we then set $E_{\infty} = \cap_{j\geq 0} E_j$, which is necessarily non-empty, we conclude that $E_{\infty} \subset J(h^mg)$. Consequently, $E \cap J(h^mg) \neq \emptyset$.

We now finish as above. Cover $g^{-1}(J(h))$ by finitely many balls $B_i$, for $i=1,\ldots, k$, of radius $r>0$ so that $g(B_i)$ intersects $J(h)$. Choose $M_3$ large enough so that if $m\geq M_3$, then $E \subset (h^mg)(B_i)$ for $i=1,\ldots, k$. By complete invariance, we conclude that $J(h^mg)$ meets each $B_i$. If we are therefore given $\epsilon>0$, we can find $M\in \N$ large enough so that if $m\geq M$ then $g^{-1}(J(h)) \subset N_{\epsilon}(J(h^mg))$.
\end{proof}

We are now in a position to prove the following characterization of $J(G)$.

\begin{proposition}
\label{prop:8}
We have \lefteqn{ J(G) = \overline{ \bigcup_{g\in G} J(g) }. }
\end{proposition}

In the case of rational semigroups, this result is proved (see \cite{HM}) by using the density of repelling fixed points of elements of $G$ in $J(G)$. In our situation, however, while it is known that periodic points of a uqr map are dense in its Julia set (see \cite{Siebert}), it is still open as to whether the repelling periodic points are dense. We will instead use Proposition \ref{prop:6} and Lemma \ref{lem:prop8}.

\begin{proof}[Proof of Proposition \ref{prop:8}]
For ease of notation, denote by $X$ the set $\overline{ \bigcup_{g\in G} J(g) }$. Clearly $X$ is closed. Moreover, $X$ cannot be any smaller and still contain $J(G)$ since $J(g) \subseteq J(G)$ for each $g\in G$. We then are done, by Proposition \ref{prop:6}, if we can show that $X$ is backward invariant. By Lemma \ref{lem:prop8}, $J(h^mg)$ accumulates on $g^{-1}(J(h))$. Since $h^mg \in G$ and $X$ is closed, it follows that $g^{-1}(J(h)) \subset X$ and we are done.
\end{proof}

The main aim now is to show that $J(G)$ is uniformly perfect, provided that $G$ satisfies a uniform H\"older condition. The idea of the proof combines ideas of Stankewitz \cite{Stank} for rational semigroups, and the author and Nicks \cite{FN} for uqr maps. First, we make a definition.

\begin{definition}
Let $\alpha >0$. Then the class of uniformly $\alpha$-H\"older maps is
\[\left \{ f:S^n \to S^n :  \operatorname{Lip}_{\alpha}(f) = \sup_{x,y \in S^n} \frac{ \chi (f(x) , f(y) )}{\chi( x,y)^{\alpha} } < \infty \right \}.\]
\end{definition}

We next show that $K$-quasiregular maps of the sphere are uniformly H\"older continuous.

\begin{lemma}
\label{lem:lip}
If $f:S^n \to S^n$ is $K$-quasiregular and non-constant, then $\operatorname{Lip}_{1/K}(f) < \infty$.
\end{lemma}

\begin{proof}
Suppose the claimed result is not true. Then we can find sequences $x_m,y_m$ in $S^n$ with 
\[ \frac{ \chi ( f(x_m) , f(y_m) )}{\chi (x_m,y_m)^{1/K} } \to \infty\]
as $m\to \infty$. Since $S^n$ is bounded in the chordal metric, we must have $\chi(x_m,y_m) \to 0$. 
By passing to a subsequence if necessary, we may assume that $x_m$ and $y_m$ converge to some point $x_0 \in S^n$.

Choose M\"obius maps $A_m,B_m$ that are chordal isometries and which map $x_m$ and $f(x_m)$ to $0$ respectively. Denote by $f_m$ the $K$-quasiregular map $B_m \circ f \circ A_m^{-1}$. We may choose $r>0$ small enough so that  on $B_{\chi}(0,r)$, each $f_m$ omits a fixed neighbourhood of infinity. Then by Theorem \ref{mini}, $\{ f_m |_{B_{\chi}(0,r)} \}$ is a normal family. 

By passing to a subsequence and re-labelling, we may assume that $f_m \to f_0$ locally uniformly on $B_{\chi}(0,r)$. The limit $f_0$ is non-constant, since otherwise $f$ would be constant in a neighbourhood of $x_0$. By \cite[Theorem III.4.7]{Rickman}, there exist constants $C,s>0$ so that
\begin{equation}
\label{eq:lip1} 
\chi ( f_0(x), 0) \leq C \chi (x,0)^{1/K},
\end{equation}
for $\chi(x,0 )<s$.
Set $r_1 < \min \{ r/2,s/2 \}$.
Since $f_m \to f_0$ uniformly on $B_{\chi}(0,r_1)$, given $\epsilon >0$, there exists $M\in \N$ so that if $m\geq M$ then we have
\begin{equation}
\label{eq:lip2}
\chi( f_m(x) , f_0(x) ) < \epsilon
\end{equation}
for all $x\in B_{\chi}(0,r_1)$. Then by \eqref{eq:lip1} and \eqref{eq:lip2}, we have for $m\geq M$
\begin{align*}
\chi( f(x_m), f(y_m) ) &= \chi( f_m(A_m(x_m)) , f_m(A_m(y_m)) ) \\
&= \chi ( f_m(0) , f_m(A_m(y_m)) ) \\
&\leq \chi ( f_m(0) , f_0(0) ) + \chi ( f_0(0) , f_0(A_m(y_m)) ) + \chi ( f_0(A_m(y_m)) , f_m( A_m(y_m) ) ) \\
& < 2\epsilon +C\chi(0 , A_m(y_m ) )^{1/K} \\
& = 2\epsilon + C\chi(x_m,y_m)^{1/K}.
\end{align*}
Since we can make $\epsilon$ as small as we like by choosing $m$ large enough, we obtain a contradiction.
\end{proof}

We now prove the main result of this section.

\begin{theorem}
\label{thm:unifperf}
Suppose $G$ is a $K$-quasiregular semigroup generated by $\{ g_i : i \in I\}$ and there is a constant $C>0$ so that $\operatorname{Lip}_{1/K}(g_i) \leq C$ for all $i \in I$. Then $J(G)$ is uniformly perfect.
\end{theorem}

\begin{proof}
Suppose the result is false and $J(G)$ is not uniformly perfect.
Fix $q$ distinct points in $J(G)$, where $q = q(n,K)$ is Rickman's constant, and choose $\delta >0$ small enough so that any two of these points are at a chordal distance greater than $C\delta^{1/K}$ from each other and $\delta < d/2$, where $d$ denotes the chordal diameter of $J(G)$.

Since $J(G)$ is assumed to not be uniformly perfect, by Lemma \ref{lem:chordal} there is a sequence of round ring domains $A_m$ centred at $x_m \in J(G)$, which lie in $F(G)$, separate $J(G)$ and with $\mod(A_m) \to \infty$. If $E_m$ is the component of $S^n \setminus A_m$ with smaller chordal diameter, then $\diam(E_m) \to 0$. We may assume that $\diam(E_m) < \delta$ for all $m$.

Since $E_m$ contains elements of $J(G)$, then by Propsition \ref{prop:8}, there exists $h \in G$ so that $E_m \cap J(h) \neq \emptyset$. If $\mathcal{E}(h) \cap J(G) = \emptyset$, let $U = \emptyset$. Otherwise, let $U$ be the union of at most $q-1$ open chordal disks centred at points of $\mathcal{E}(h) \cap J(G)$, each of diameter at most $\frac{d}{2(q-1)}$ so that the boundary of each disk in $U$ contains a point of $J(G)$. We can find such disks since $J(G)$ is perfect by Proposition \ref{prop:7}. 

Let $J' = J(G) \setminus U$. Then $\diam (J') \geq \frac{d}{2} > \delta$.
Moreover,  by Lemma \ref{lem:blowup}, $J' \subset h^k(E_m)$ for some integer $k$. In particular, $\diam (h^k(E_m)) > \delta$. Let $f_m \in G$ be an element of minimal word length (in terms of the generators of $G$) so that $\diam ( f_m(E_m)) > \delta$. We can write $f_m = g_{i_1} \circ \ldots g_{i_j}$, where the indices depend on $m$. Set $F_m = g_{i_2} \circ \ldots \circ g_{i_j}$, or set $F_m$ equal to the identity if $f_m$ is one of the generators. Then by the minimality of the choice of $f_m$, we have $\diam( F_m(E_m)) \leq \delta$. Hence
\begin{align*}
\diam(f_m(E_m)) & = \diam ( g_{i_1}(F_m(E_m))) \\
&\leq C\diam(F_m(E_m))^{1/K}\\
&\leq C\delta^{1/K}.
\end{align*}

Since $A_k \cup E_k$ is a ball, we can find M\"obius maps $\psi_k : \B^n \to A_k \cup E_k$. Then $\psi_k(0) \in J(G)$. Let $V_k = \psi_k^{-1}(E_k)$ and observe that since $\mod(\B^n \setminus V_k)  = \mod A_k$, we have $\diam( V_k) \to 0$.

Let $h_m = f_m \circ \psi_m$. Then $h_m( \B^n \setminus V_k) \subset F(G)$ and 
\[ \diam ( h_m(V_m)) = \diam ( f_m(E_m)) \leq C\delta^{1/K}. \]
Hence each $h_m(\B^n)$ omits at least $q$ of the points $x_1,\ldots, x_{q+1}$ and so by Theorem \ref{mini}, the family $\{ h_m \}$ is normal in $\B^n$. By equicontinuity of this family, we have $\diam(h_m(V_m)) \to 0$ since $\diam(V_m) \to 0$ and $0 \in V_m$ for all $m$. This contradicts $\diam(h_m(V_m))  = \diam (f_m(E_m)) > \delta$.
\end{proof}

\begin{corollary}
\label{cor:1}
If $G$ is a finitely generated quasiregular semigroup, then $J(G)$ is uniformly perfect.
\end{corollary}

\begin{proof}
This follows immediately from Lemma \ref{lem:lip} and Theorem \ref{thm:unifperf}.
\end{proof}

In \cite{Stank}, it was shown that for a rational semigroup $G$ whose generators are all uniformly Lipschitz, $J(G)$ is uniformly perfect. Moreover, in \cite{HMM} it was shown that near a fixed point of a uqr mapping $f$, where $f$ is locally injective, $f$ is in fact locally bi-Lipschitz. It is therefore worth pointing out that uqr mappings need not be locally bi-Lipschitz away from branch points.
We will give a simple example to illustrate this.

\begin{example}
Let $f(z) = z^2$ be define on the unit disk $\D$ and let $U \subset \D \setminus \{ 0 \}$ be a domain on which $f$ is injective and so that $O^+(U) \cap U = \emptyset$. Let $z_0 \in U$ and choose a small neighbourhood $V$ of $z_0$ with $\overline{V} \subset U$ 
so that we can define on $V$
\[ g(z) = f(z_0) + (z-z_0)|z-z_0|^{\beta},\]
for $\beta \in (-1,0)$. We moreover require $\overline{g(V)} \subset f(U)$. We then replace $f$ on $V$ by $g$ and interpolate on $U \setminus \overline{V}$ via Sullivan's Annulus Theorem (see \cite{TV}) to obtain a map $f_1:\D \to \D$. By construction, $f_1$ is uqr, but $\operatorname{Lip}_1(f_1)$ is not finite due to the behaviour of $f_1$ at $z_0$.
\end{example}

\section{Examples}

In this section, we will exhibit several classes of examples of quasiregular semigroups. However, we start by noting that we have to be careful when generating quasiregular semigroups. Throughout, we have made the assumption that the degree of each element of a quasiregular semigroup is at least two. Unlike with rational semigroups, we cannot just include a M\"obius map in the generating set without taking care.

\begin{proposition}
\label{prop:uqr}
Let $f:S^n \to S^n$ be $K$-uqr with $K>1$. Suppose there exists $x_0 \in S^n$ with $f'(x_0) \neq \lambda \mathcal{O}$ for some $\lambda >0$ and orthogonal matrix $\mathcal{O}$. Then there exists a M\"obius map $A$ so that $A\circ f$ is not uqr. Consequently $\left < f,A \right >$ is not a quasiregular semigroup.
\end{proposition}

We will use the fact that the only uniformly quasiconformal linear maps are of the form $\lambda \mathcal{O}$ for $\lambda >0$ and $\mathcal{O}$ orthogonal, see \cite{Martin}.

\begin{proof}
We may assume $x_0, f(x_0) \in \R^n$, otherwise conjugate by an appropriate M\"obius map. The point $x_0$ cannot be a fixed point of $f$ because then the maximal dilatation of $[f'(x_0) ]^m$ diverges as $m\to \infty$, see \cite{Martin}.
Let $A$ be a translation sending $f(x_0)$ to $x_0$. Then $A\circ f$ has $x_0$ as a fixed point, and $ [ (A\circ f)^m ] '(x_0) = [ f'(x_0) ]^m$. Again we apply \cite{Martin} to conclude that $A\circ f$ is not uqr.
\end{proof}

\subsection{Solutions to the Schr\"oder equation}

We start this section with the following definition, see \cite{IMbook} and for a more general definition see \cite{FM}.

\begin{definition}
A quasiregular mapping $h:\R^n \to \R^n$ is called {\it strongly automorphic} if there exists a discrete group of isometries $G$ so that the following two conditions hold:
\begin{enumerate}[(i)]
\item $h\circ g = h$ for all $g\in G$,
\item $G$ acts transitively on the fibres $h^{-1}(y)$, that is, if $h(x_1) = h(x_2)$, then there exists $g\in G$ such that $x_2=g(x_1)$.
\end{enumerate}
\end{definition}

Then $G$ has a maximal translation subgroup $\mathcal{T}$. There are three classes (see \cite{FM} for more details):
\begin{enumerate}[(i)]
\item $\wp$-type, where $\mathcal{T}$ is isomorphic to $\Z^n$,
\item sine-type, where $\mathcal{T}$ is isomorphic to $\Z^{n-1}$ and there is a rotation identifying prime ends of the beam $\R^n / \mathcal{T}$,
\item Zorich-type, where $\mathcal{T}$ is isomorphic to $\Z^{n-1}$ and there is no such rotation.
\end{enumerate}

If $A = \lambda \mathcal{O}$, for $\lambda >1$ and $\mathcal{O}$ orthogonal, satisyfing $AGA^{-1} \subset G$, then the Schr\"oder equation
\begin{equation}
\label{eq:sch} 
f\circ h = h\circ A
\end{equation}
has a unique solution $f$ that is a uqr map that extends to a uqr map $f:S^n \to S^n$, see \cite{FM,IMbook}.
Again there are three possibilities:
\begin{enumerate}[(i)]
\item $f$ is of Latt\'es-type if $h$ is of $\wp$-type,
\item $f$ is of Chebyshev-type if $h$ is of sine-type,
\item $f$ is of power-type if $h$ is of Zorich-type.
\end{enumerate}

Mayer \cite{Mayer1,Mayer2} waas the first to construct example of power-type and Chebyshev-type  uqr mappings.
We remark that each of these three cases can occur in every dimension at least two. Now, fix a Zorich-type map $h$ and assume that the translation part of $G$ is generated by $e_1,\ldots,e_{n-1}$, where $e_j$ is the standard $j$'th unit vector in $\R^n$.
We will also assume that $h$ maps the hyperplane $\{x_n = r \}$ onto the sphere $\{ |x| = e^r \}$ as in the standard constuction of a Zorich-type map (see \cite{FM} and references therein).

Now, any $g\in G$ has the form
\[ g(x_1,\ldots,x_{n-1}, x_n) = ( x_1+m_1,\ldots, x_{n-1}+m_{n-1}, x_n )\]
for some integers $m_1,\ldots, m_{n-1}$. For $d\in \Z$ and $\lambda >0$, define
\[A_{d,\lambda}(x_1,\ldots, x_{n-1},x_n) = ( dx_1, \ldots, dx_{n-1}, dx_n  + \ln \lambda).\]
A calculation shows that
\[ A_{d,\lambda}gA_{d,\lambda}^{-1}(x_1,\ldots, x_{n-1},x_n) = (x_1 + dm_1 ,\ldots, x_{n-1} + dm_{n-1}, x_n)\]
and hence $A_{d,\lambda}gA_{d,\lambda}^{-1} \in G$. Denote by $f_{d,\lambda}$ the unique uqr solution to the Schr\"oder equation.

We observe that in dimension two, if $h$ is the usual exponential function, then $f_{d,\lambda}(z) = \lambda z^d$. We will denote by $\mathcal{F}$ the family
\[ \mathcal{F} = \{ f_{d,\lambda} : d\in \Z, \lambda >0 \}.\]
Now, if we have two maps $f_1,f_2 \in \mathcal{F}$, then they arise from two linear maps $A_{d_1,\lambda_1}$ and $A_{d_2,\lambda_2}$ and solving the Schr\"oder equation. Hence
\[ ( f_1 \circ f_2) \circ h = h \circ ( A_{d_1,\lambda_1} \circ A_{d_2,\lambda_2} ).\]
Since 
\[  A_{d_1,\lambda_1} \circ A_{d_2,\lambda_2} = A_{d_1d_2, \lambda_1\lambda_2^{d_1}},\]
it follows that $f_1\circ f_2 \in \mathcal{F}$ and hence $\mathcal{F}$ is closed under composition. Moreover, since each $A_{d,\lambda}$ is conformal, if $h$ is $K$-quasiregular, then each element of $\mathcal{F}$ is $K^2$-quasiregular. We have shown the following result.

\begin{theorem}
\label{thm:powertype}
If $G = \left < \{  f_{d_i, \lambda_i} \} _{i\in I} \right >$, then $G$ is a quasiregular semigroup.
\end{theorem}

One can make similar constructions of quasiregular semigroups using sine-type maps and $\wp$-type maps by using solutions of the Schr\"oder equation with $A_d(x) = dx$. We, however, will restrict to the case above and give some results on the types of Julia sets that can arise from such semigroups.

\begin{lemma}
\label{lem:juliasphere}
If $d\in \Z$ and $\lambda >0$, then $J(f_{d,\lambda}) = S(\lambda^{1 / (1-d)})$, where $S(r)$ denotes the Euclidean sphere centred at $0$ of radius $r>0$.
\end{lemma}

\begin{proof}
Since $h$ maps the hyperplane $\{x_n = r\}$ onto $S(e^r)$, it follows from the Schr\"oder equation that $f_{d,\lambda}$ maps $S(r)$ onto $S(\lambda r^d)$. Since $0,\infty$ are superattracting fixed points for $f_{d,\lambda}$ (see  \cite{FM}), the Julia set of $f_{d,\lambda}$ arises as the sphere with radius $t$, where $t = \lambda t^d$. The result follows.
\end{proof}

Our first example is the higher dimensional analogue of \cite[Example 1]{HM}, which in particular shows that the Julia set of a quasiregular semigroup can have interior points without being all of $S^n$, in contrast to the case of uqr mappings.

\begin{proposition}
For every $a>1$, there exists a quasiregular semigroup $G$ in $\R^n$, $n\geq 2$, so that $J(G)$ is the closed ring $\{x:1\leq|x|\leq a\}$.
\end{proposition}

\begin{proof}
Let $f = f_{2,1}$ and $g = f_{2,1/a}$ where $a>1$. Then $f,g \in \mathcal{F}$ and denote by $G = \left < f,g \right >$. Then $J(f) = S(1)$ and $J(g) = S(a)$ by Lemma \ref{lem:juliasphere}. Since $f^{-1}(S(r)) = S(\sqrt{r})$ and $g^{-1}(S(r)) = S(\sqrt{ar})$, we have
\[ f^{-1}(J(g)) = S(\sqrt{a}) = g^{-1}(J(f)),\]
and
\[ f^{-1}(S(\sqrt{a})) = S(a^{1/4}), \quad g^{-1}( S(\sqrt{a}) ) = S(a^{3/4}).\]
By induction and backward invariance, $S(a^{j/2^m}) \subset J(G)$ for $j=1,\ldots, 2^m-1$. Since $J(G)$ is closed, it follows that $ \{ 1\leq|x| \leq a\}\subset J(G)$.
Now, $\{|x| <1 \}$ is mapped into itself by both $f$ and $g$, and the same is true for $ \{ |x| > a \}$. Hence both these sets are contained in $F(G)$. It follows that $J(G) = \{1 \leq |x| \leq a \}$.
\end{proof}

The next example is a higher dimensional analogue of a result of Morosawa \cite{Morosawa} which shows that in dimension $n$, we can construct a quasiregular semigroup $G$ with Hausdorff dimension arbitrarily close to $n$. We remark that the Hausdorff dimension of $J(G)$ is strictly positive, since $\dim_H J(g) > 0$ by \cite[Theorem 1.2]{FN}.

\begin{proposition}
There exists a quasiregular semigroup $G$ in $\R^n$, for $n\geq 2$, so that $J(G)$ is, topologically, a product of a Cantor set and an $(n-1)$-sphere. Moreover, for any $\epsilon >0$, $G$ can be chosen so that $\dim _H J(G) > n-\epsilon$.
\end{proposition}

In \cite{Morosawa} and in dimension two, such Julia sets are called Cantor targets. In higher dimensions, we might call them Cantor shells.

\begin{proof}
For $k\in \N$, let $p_k = f_{2^k,2^{2-2^k}} \in \mathcal{F}$ and $q_k = f_{2^k,2^{1-2^k}} \in \mathcal{F}$. By Lemma \ref{lem:juliasphere}, we have $J(p_k) = S(2^{c_k})$, where $c_k = \frac{2^k-2}{2^k-1}$ and $J(q_k) = S(2)$.

For $N\in \N$, let $G = \left < p_1,\ldots, p_{N-1}, q_N \right >$. Then it is not hard to see that $\{ |x| <1 \} \cup \{ |x| >2 \} \subset F(G)$. Next, consider the logarithm to base $2$ of the inverses of the radial parts of $p_k$ and $q_N$, that is
\[ \varphi_k(t) = \frac{t}{2^k}+1 -\frac{1}{2^{k-1}}, \quad \phi_N(t) = \frac{t}{2^N} +1 - \frac{1}{2^n}.\]
One can then show, as in \cite{Morosawa}, that the iterated function system (IFS) generated by $\{ \varphi_1, \ldots, \varphi_{N-1}, \phi_N \}$ satisfies Moran's open set condition which makes the Hausdorff dimension of the attractor set computable. For $N$ chosen large enough, one can show that $J(G)$ is arbitrarily close to $n$. Again, we refer to \cite[Theorem 6]{Morosawa} for the details of this argument.
\end{proof}

\subsection{Antoine's necklaces}

We briefly recall the uqr map constructed in \cite{FW}, which has a particular wild Cantor set, called an Antoine's necklace, for its Julia set. The construction is in dimension three and it is still open as to whether there exist analogous uqr constructions in dimension at least four.

Let $X_0 \subset B(0,2) \subset \R^3$ be a solid torus and let $m\geq 4$ be a square even integer, say $m=d^2$. Choose mutually disjoint solid tori $X_{1,1},\ldots, X_{1,m}$ contained in the interior of $X_0$ so that $X_{1,i}$ and $X_{1,j}$ are linked if and only if $|i-j| \equiv 1 (\operatorname{mod} m)$ and, when linked, they form a Hopf link. Let $\varphi_j :X_0 \to X_{1,j}$ be linear maps, for $j=1,\ldots, m$ and define
\[ X_1 = \bigcup_{j=1}^m X_{1,j} = \bigcup_{j=1}^m \varphi_j(X_0).\]
The required uqr map $f$ is constructed as follows:
\begin{enumerate}[(i)]
\item on $X_1$, $f$ is defined via the various linear maps $\varphi_j^{-1}$,
\item $f$ is a particular branched covering $\overline{X_0 \setminus X_1} \to \overline{B(0,2) \setminus X_0}$,
\item outside $B(0,2)$, we set $f$ to be a uqr power-type map $f_{d,1}$ of degree $d^2$, recalling the previous section,
\item finally on $B(0,2) \setminus \operatorname{int}(X_0)$, we use an extension theorem of Berstein and Edmonds.
\end{enumerate}
We refer to \cite{FW} for more details. The Julia set for $f$ is then the attractor for the IFS generated by the maps $\varphi_1,\ldots,\varphi_m$ and is a wild Cantor set. If an orbit remains on $J(f)$, then $f$ only ever acts by M\"obius maps. Otherwise, $f$ acts by finitely many M\"obius maps, two quasiregular maps, and then a uqr map, with some of these steps possibly omitted. It follows that $f$ is uqr.

This construction is far from unique, since $m$ just needs to be a sufficiently large even square integer, and the exact configuration of the smaller tori within the larger gives plenty of freedom. Let us then fix $X_0$ and consider a collection $f_1,\ldots, f_k$ of uqr maps constructed as above.

\begin{theorem}
\label{thm:antoine}
Let $G = \left < f_1,\ldots, f_k \right>$ act on $S^3$, with $f_i$ as above. Then $G$ is a quasiregular semigroup and $J(G)$ is the attractor set for the IFS generated by the collection $\varphi^i_j$ of linear maps used to construct $f_i$.
\end{theorem}

\begin{proof}
Each $f_i$ has degree $m_i$, where $m_i = d_i^2$ is a sufficiently large square even integer, and the Julia set of $f_i$ is the attractor for the IFS generated by the linear maps $\varphi^i_1,\ldots, \varphi^i_{m_i}$. Denote by $X_1^i$ the set $X_1$ in the above construction for $f_i$.

On $S^3 \setminus B(0,2)$, $G$ is a family of power-type maps generated by $f_{d_i,1}$ and hence by Theorem \ref{thm:powertype} and Lemma \ref{lem:juliasphere}, $S^3 \setminus B(0,2) \subset F(G)$. Next, let $x_0 \in B(0,2) \setminus \bigcup_{i\in I} X_1^i$. Then no matter what elements of $G$ act on $x_0$, after two iterations, $x_0$ is mapped outside $B(0,2)$. Otherwise, $x_0$ is acted on by M\"obius maps, either infinitely often or finitely many times before being mapped into $B(0,2) \setminus \bigcup_{i\in I} X_i^i$. It follows that $G$ is a quasiregular semigroup.

The claim about $J(G)$ follows, since if $x_0$ is not in the attractor set for $\varphi_j^i$, then we can find $\delta >0$ so that $O^+(B(x_0,\delta))$ omits the infinitely many points in, say, $J(f_1)$. This means that $x_0 \in F(G)$.
\end{proof}

\subsection{Conformal traps}

Conformal traps were introduced in \cite{IM} and yielded the first examples of uqr maps constructed in $S^n$ for $n\geq 3$. One of the features such uqr maps have is that the Julia set is a tame Cantor set. We briefly recall how one can modify a given quasiregular map to obtain a uqr map with a conformal trap following the presentation in \cite{MP}.

Starting with any non-injective quasiregular map $f:S^n \to S^n$, for $n\geq 2$, and of degree $d$, choose $x_0 \in S^n$ satisfying the following properties:
\begin{enumerate}[(i)]
\item $\{ x_0, f(x_0), f^{-1}(x_0)\} \cap ( B_f \cup \{ \infty \} )  = \emptyset$,
\item there exists a Euclidean ball $U_0 := B(x_0,r)$ so that $f^{-1}(U_0)$ has components $U_1,\ldots, U_d$ and $f:U_j \to U_0$ is injective for $j=1,\ldots, d$,
\item $f(U_0) \cap \bigcup_{i=0}^d U_i = \emptyset$.
\end{enumerate}

Let $f^{-1}(x_0) = \{x_1,\ldots, x_d\}$ and choose $a,b>0$ with $2b<a$ and
\begin{enumerate}[(i)]
\item $B(x_i ,a) \subset U_i$ for $i=0,\ldots, d$,
\item $B(f(x_0),a) \subset f(U_0)$,
\item $B(x_0,b) \subset \bigcap_{i=1}^d f(B(x_i,a))$,
\item $f(B(x_0,b)) \subset B(f(x_0),a)$.
\end{enumerate}

On $S^n \setminus \bigcup_{i=0}^d B(X_i,a)$, we now set $\widetilde{f} = f$. For $i=1,\ldots, d$ set $\widetilde{f} |_{B(x_i,b)}$ to be a translation onto $B(x_0,b)$, and set $\widetilde{f}|_{B(x_0,b)}$ to be a translation onto $B(f(x_0,b))$. On each $B(x_i,a) \setminus B(x_i,b)$, apply Sullivan's Annulus Theorem to find a quasiconformal extension defining $\widetilde{f}$ everywhere.

Next let $\Phi :S^n \to S^n$ be a M\"obius involution exchanging $B(x_0,b)$ with its complement and then define $g = \Phi \circ \widetilde{f}$. The map $g :S^n \to S^n$ is uqr. To see this, on $B(x_0,b)$, $\widetilde{f}$ is a translation onto $B(f(x_0),b)$ and $\Phi$ is then a conformal map. Hence $g$ conformally maps $B(x_0,b)$ into itself. Hence $g$ has an attracting fixed point in $B(x_0,b)$. On $S^n \setminus \bigcup_{i=1}^d B(x_i,a)$, $\widetilde{f}$ may have distortion, but then $\Phi$ will map the image into $B(x_0,b)$, on which we know the iterates of $g$ act conformally. Finally, for $i=1,\ldots, d$, $g$ acts on $B(x_i,a)$ by a conformal map, say $\varphi_i$, with image $S^n \setminus \overline{B(x_0,a)}$.

The Julia set of $g$ is then the attractor set of the IFS generated by $\varphi_1^{-1},\ldots, \varphi_d^{-1}$, and can be shown to be a tame Cantor set.

\begin{theorem}
\label{thm:traps}
Suppose $g_1,\ldots, g_k$ are uqr maps in $S^n$, for $n\geq 2$, constructed as above with the same conformal trap $B(x_0,b)$, for $b>0$. Then $G = \left < g_1,\ldots, g_k \right >$ is a quasiregular semigroup and $J(G)$ is an attractor set for an IFS generated by M\"obius maps.
\end{theorem}

\begin{proof}
On $\overline{B(x_0,b)}$, $G$ acts by conformal maps, whereas if $x_0$ is outside $B(x_0,b)$, any element of $G$ either acts conformally, or maps $x_0$ into $B(x_0,b)$. We conclude that $G$ is a quasiregular semigroup.

If $J(g_i)$ is the attractor set for the IFS generated by $\varphi_{i,1}^{-1} ,\ldots, \varphi_{i,d_i}^{-1}$, then $J(G)$ is the attractor set for the IFS generated by $\bigcup_{i=1}^k \{ \varphi_{i,1}^{-1} ,\ldots, \varphi_{i,d_i}^{-1} \}$.
\end{proof}

\end{document}